\documentclass[11pt, reqno]{amsart}
\usepackage{amsmath,amssymb,amsthm}

\usepackage{parskip}
\usepackage[T1]{fontenc}
\usepackage{newtxtext, newtxmath}
\usepackage[margin=1in]{geometry}
\usepackage{enumitem}
\usepackage[pagebackref, colorlinks=true,linkcolor=red,citecolor=blue,urlcolor=blue]{hyperref}

\theoremstyle{plain}
\newtheorem{theorem}{Theorem}[section]
\newtheorem{proposition}[theorem]{Proposition}
\newtheorem{lemma}[theorem]{Lemma}
\newtheorem{corollary}[theorem]{Corollary}
\theoremstyle{definition}
\newtheorem{remark}[theorem]{Remark}
\newtheorem{notation}[theorem]{Notation}

\DeclareMathOperator{\Ind}{Ind}
\DeclareMathOperator{\Cl}{Cl}
\DeclareMathOperator{\lk}{lk}
\DeclareMathOperator{\st}{star}
\newcommand{\Z}{\mathbb{Z}}
\newcommand{\rH}{\widetilde{H}}

\title[Field Independence of the First Seven Betti Numbers of Flag Complexes]
{Field Independence of the First Seven Betti Numbers of Flag Complexes}
\author[Omkar Javadekar]{Omkar Javadekar}
	\address{Chennai Mathematical Institute, Siruseri, Tamilnadu 603103. INDIA}
	\email{omkarjavadekar@gmail.com, omkarj@cmi.ac.in}
    

\begin{document}

\begin{abstract}
In 2006, Katzman showed that the first six Betti numbers of the Stanley--Reisner ring of a flag complex are field independent. He also found flag complexes on eleven vertices whose eighth Betti number depends on the field, and asked whether the seventh is always field independent. 
We answer this affirmatively by proving a stronger, purely topological result. Let $\tau(d)$ be the least number of vertices of a flag complex whose $d$-th reduced integral homology has torsion. We prove that $\tau(d)\geq d+10$ for every $d\ge0$. This bound yields the field independence of the seventh Betti number. Equivalently, combining our result with Katzman's, for every finite simple graph $G$, the first seven Betti numbers of the edge ideal $I(G)$ are field independent.
\end{abstract}
\maketitle

\section{Introduction}

Flag complexes are simplicial complexes in which every set of pairwise adjacent vertices forms a simplex. Such complexes arise naturally in combinatorics and topology, and much of their structure can be studied through the graphs underlying them. The homology of a flag complex may contain torsion, and this torsion governs how certain algebraic invariants of the complex depend on the coefficient field. In this article, we study how many vertices a flag complex must have for torsion to appear in a given homological degree.

Let $\mathsf k$ be a field and $S=\mathsf k[x_1,\dots,x_n]$ be a standard graded polynomial ring over $\mathsf k$. Given a finite simple graph $G$ on the vertex set $\{x_1,\dots,x_n\}$ and edge set $E(G)$, the \emph{edge ideal} of $G$ is the squarefree monomial ideal $I(G)=\bigl\langle x_ix_j\mid\{x_i,x_j\}\in E(G)\bigr\rangle\subseteq S$.
The $(i,j)$-th graded Betti number of $S/I(G)$ is given by $\beta_{i,j}^{\mathsf k}(S/I(G))=\dim_{\mathsf k}\left(\operatorname{Tor}_i^S(S/I(G),\mathsf k)_j\right)$, and the $i$-th (total) Betti number is given by the sum $\beta_i^{\mathsf k}(S/I(G))=\sum_j\beta_{i,j}^{\mathsf k}(S/I(G))$.
Although $I(G)$ is defined purely in terms of the graph, its Betti numbers need not be independent of the coefficient field $\mathsf k$. Understanding this dependence leads naturally to a question about the homology of flag complexes.

The reason for this connection comes from the Stanley--Reisner correspondence. Let $\Ind(G)$ denote the independence complex of $G$, whose faces are the independent sets of $G$. Then $S/I(G)=\mathsf k[\Ind(G)]$. So, the edge ideal $I(G)$ is precisely the Stanley--Reisner ideal of $\Ind(G)$. The complexes that arise in this way are exactly the flag complexes, and conversely every flag complex is the independence complex of a finite simple graph on the same vertex set. Thus, Stanley--Reisner rings of flag complexes and quotients by edge ideals of graphs form the same class of rings, and we use the two viewpoints interchangeably. 

Hochster's formula expresses the Betti numbers of $\mathsf k[\Ind(G)]$ in terms of the reduced homology of the induced subcomplexes of $\Ind(G)$. The dimension of such a homology group  depends on the field precisely when the corresponding integral homology has torsion. In this way, the field dependence of Betti numbers is governed by torsion in the homology of flag complexes. The classical example of such field dependence is the six-vertex triangulation of $\mathbb{RP}^2$, whose first integral homology contains a copy of $\Z/2\Z$. It is the smallest simplicial complex whose homology has torsion, and the standard first example in which torsion produces field dependent behaviour of algebraic invariants (see, e.g., \cite{Reisner}).

It is natural to ask for which graphs the Betti numbers are field independent. For forests and for graphs having every vertex of degree at most two, Jacques and Katzman \cite{JacquesKatzman} showed that the Betti numbers of $I(G)$ are field independent. 
For chordal graphs, field independence was proved by H\`a and Van Tuyl \cite{HaVanTuyl}. The same holds whenever the edge ideal has a linear resolution, which by Fr\"oberg's theorem \cite{Froberg} happens exactly when the complement of the graph is chordal. In fact, the linear strand is always field independent, since $\beta_{i,i+1}^{\mathsf k}(S/I(G))$ is determined by reduced homology in degree zero, which is always torsion-free. 
On the other hand, Dalili and Kummini constructed a connected bipartite graph whose edge ideal has field dependent Betti numbers \cite{DaliliKummini}. So, field independence of Betti numbers does not hold for all graphs.

Instead of fixing the class of graphs, one may fix the homological degree and ask how large $i$ can be for $\beta_i^{\mathsf k}(S/I(G))$ to be field independent for all graphs $G$. Terai and Hibi \cite{TeraiHibi} showed that the third and fourth Betti numbers are always field independent, and Katzman \cite{Katzman} extended this to the fifth and sixth. Katzman further proved that all Betti numbers are field independent for every graph on at most ten vertices. On eleven vertices, he found exactly four graphs up to isomorphism with field dependent Betti numbers, and in each of these examples the dependence occurs only in $\beta_8$ and $\beta_9$. This left open whether $\beta_7$ is always field independent, and Katzman asked precisely this question. He also observed that if $\beta_7$ were field dependent for some graph, then there would have to be such an example on twelve vertices.

The aim of this article is to answer Katzman's question. We do it by proving a stronger statement about the presence of torsion in the homology of flag complexes. For $d\geq 0$, let $\tau(d)$ denote the least number of vertices of a flag complex $X$ such that $\rH_d(X;\Z)$ contains torsion, with $\tau(d)=\infty$ if no such complex exists. Our main theorem is the following.

\begin{theorem}[see Theorem~\ref{thm:main}]\label{thm:intro-main}
For every $d\ge0$, we have $\tau(d)\geq d+10$.
\end{theorem}
Theorem~\ref{thm:intro-main} is of independent interest, as it gives a lower bound on the number of vertices required for torsion to occur in the integral homology of a flag complex, without reference to edge ideals or Betti numbers. We emphasize this point because the theorem applies in a broader topological setting than the one needed for our algebraic application. Using Theorem~\ref{thm:intro-main} along with the relationship between torsion in integral homology and the field dependence of Betti numbers, we obtain the following consequence for edge ideals.

\begin{corollary}[see Corollary~\ref{cor:beta7}]\label{cor:intro-beta7}
For every graph $G$ and every field $\mathsf k$, the Betti number $\beta_7^{\mathsf k}(S/I(G))$ is
independent of $\mathsf k$.
\end{corollary}

Together with Katzman's result for $i\le6$, this shows that the first seven Betti numbers of $S/I(G)$ are field independent for every finite simple graph $G$. This result is best possible, since Katzman's  $11$-vertex examples have field dependent eighth Betti number. Therefore, seven is the largest initial segment of Betti numbers that is field independent for every graph. The same examples of Katzman also give $\tau(1)=11=1+10$. Hence, the bound we obtain in Theorem~\ref{thm:intro-main} is attained in degree one and the constant $10$ cannot be increased. 

The idea of the proof of Theorem~\ref{thm:intro-main} is as follows. We assume that the proposed bound fails and choose a counterexample with the least number of vertices. Using a result of Engström~\cite{Engstrom} and the behaviour of independence complexes under suspension, we show that every vertex has degree at least two. A Mayer--Vietoris argument, together with a result of Terai and Hibi~\cite{TeraiHibi} gives an upper bound on the degrees of the vertices. These bounds leave only four possible pairs $(d,m)$. Finally, we rule out three of these four cases using the preceding results, and the remaining case by a direct counting argument.

The article is organized as follows. In Section~\ref{sec:torsion}, we reduce the field independence
of Betti numbers to torsion in flag complexes and introduce the function $\tau(d)$. In
Section~\ref{sec:tools}, we collect the homological and combinatorial results about flag complexes
needed in the proof. Finally, in Section~\ref{sec:main}, we prove the main theorem and deduce the
field independence of the seventh Betti number.

\section*{Acknowledgements}
The author acknowledges support from a Postdoctoral Fellowship at the Chennai Mathematical Institute and additional support from the Infosys Foundation.

\section{Reduction to torsion in flag complexes}\label{sec:torsion}

In this section, we reduce the question of when Betti numbers are independent of the field to a question about torsion in reduced homology. We begin by recalling the necessary facts about Betti numbers, reduced homology, and flag complexes, and introduce the notation and terminology that will be used throughout the paper. We then define $\tau(d)$ to be the least number of vertices of a flag complex whose $d$-th reduced homology carries torsion. This quantity captures the obstruction to field independence and allows us to reformulate Katzman’s question as a question about $\tau(d)$.

\begin{notation}
Let $G=(V,E)$ be a finite simple graph. For $v\in V$, we write $N_G(v)$ for the \emph{open neighbourhood} of $v$, i.e.,
$N_G(v)=\{u\in V\mid\{u,v\}\in E\}$.\\
The \emph{closed neighbourhood} of $v$ is $N_G[v]=N_G(v)\cup\{v\}$, and the degree of $v$ is $\deg_G(v)=|N_G(v)|$.\\ We denote by $\overline{G}$ the \emph{complement} of $G$, i.e., the graph on $V$ in which two distinct vertices are adjacent if and only if they are not adjacent in $G$.\\
A \emph{clique} in $G$ is a set of pairwise adjacent vertices, and the \emph{clique complex} of $G$, denoted by $\Cl(G)$, is the simplicial complex whose faces are precisely the cliques of $G$.\\
The \emph{independence complex} of $G$ is $\Ind(G)=\Cl(\overline{G})$, whose faces are the independent sets of $G$. In particular, if $I(G)$ denotes the edge ideal of $G$ in $S=\mathsf{k}[x_v\mid v\in V]$, then $S/I(G)=\mathsf{k}[\Ind(G)]$.\\
A simplicial complex is called \emph{flag} if it is the clique complex of some graph. Equivalently, a simplicial complex is flag if and only if all of its minimal non-faces have cardinality two. Thus, the flag complexes are precisely the clique complexes $\Cl(G)$ of graphs $G$, or, equivalently, the independence complexes $\Ind(G)$ of graphs $G$, since $\Ind(G)=\Cl(\overline{G})$.\\
For a simplicial complex $\Delta$ on vertex set $V$ and $W\subseteq V$, we write $\Delta|_W$ for the \emph{induced
subcomplex} on $W$, i.e., we have $\Delta|_W=\{F\in\Delta\mid F\subseteq W\}$.
For $v\in V$, we let $\Delta\setminus v$ denote $\Delta|_{V\setminus\{v\}}$.\\
The \emph{star} of a vertex $v$ in a simplicial complex $X$ is the simplicial complex $\st_X(v)=\{\sigma\in X\mid \sigma\cup{v}\in X\}$. The \emph{link} of $v$ in $X$ is the simplicial complex $\lk_X(v)=\{\sigma\in X\mid v\notin\sigma,\ \sigma\cup{v}\in X\}$.

All homology groups considered in the article are reduced simplicial homology with coefficients in $\mathsf{k}$ or $\mathbb Z$. We use $\beta_i^{\mathsf{k}}(-)$ to denote the $i$-th Betti number, computed over the field $\mathsf{k}$. 
\end{notation}

\begin{remark}\label{eq:basic}
We record two facts about $\Ind(G)$ that will be used later on in the article, both of which are straightforward to verify by definition.
\begin{enumerate}
    \item[{\rm (a)}] $\Ind(G)\setminus v=\Ind(G\setminus v)$.
    \item[{\rm (b)}] $\lk_{\Ind(G)}(v)=\Ind\bigl(G\setminus N_G[v]\bigr).$
\end{enumerate}
\end{remark}

We recall the following fundamental result, which relates the Betti numbers of a Stanley–Reisner ring to the reduced homology groups of induced subcomplexes.
\begin{remark}[Hochster's formula, see {\cite{Hochster}}, {\cite[Theorem~1.1]{Katzman}}]\label{rem:hochster}
Given any field $\mathsf k$ and a simplicial complex $\Delta$ on $V$,
\[
\beta_i^\mathsf k\bigl(\mathsf k[\Delta]\bigr)=\sum_{W\subseteq V}\dim_{\mathsf k} \left(\rH_{\,\vert W\vert-i-1}\bigl(\Delta|_W;\mathsf k\bigr)\right).
\]
\end{remark}

We next recall the universal coefficient theorem and its consequence for the dimensions of homology groups over different fields.

\begin{remark}\label{rem:uct}
Let $\Delta$ be a simplicial complex and $\rH_i(\Delta;\Z)\cong\Z^{r_i}\oplus T_i$, where $T_i$ is the torsion part. Let $t_i(p)$ denote the number of cyclic summands of $T_i$ whose order is divisible by $p$. Then 
\begin{enumerate}
    \item[{\rm (a)}]  $\dim_{\mathbb Q}\left(\rH_i(\Delta;\mathbb Q)\right)=r_i$
    \item[{\rm (b)}] Given any prime $p$, we have $\dim_{\mathbb F_p}\left(\rH_i(\Delta;\mathbb F_p)\right)=r_i+t_i(p)+t_{i-1}(p)$. This follows from the well-known universal coefficient theorem for homology (see, e.g., \cite[Theorem 3A.3]{Hatcher}).
    \item[{\rm (c)}] By (a) and (b) above, given any field $\mathsf k$, we have $\dim_{\mathsf k}\left(\rH_i(\Delta;\mathsf k)\right)\geq \dim_{\mathbb Q}\left(\rH_i(\Delta;\mathbb Q)\right)$.
    \item[{\rm (d)}] From (a), (b), and (c) above, $\dim_{\mathsf k}\left(\rH_i(\Delta;\mathsf k)\right)$ is independent of $\mathsf k$ if and only if $T_i=T_{i-1}=0$. In this case, we have  $\dim_{\mathsf k}\left(\rH_i(\Delta;\mathsf k)\right)= r_i$.
\end{enumerate}
\end{remark}

The preceding remark allows us to translate the field independence of homology dimensions into a torsion-freeness condition. Combining this with Hochster's formula gives the following criterion.

\begin{proposition}\label{prop:crit}
Let $\Delta$ be a simplicial complex on $V$ and $i\ge1$. Then $\beta_i^\mathsf k(\mathsf k[\Delta])$ is independent of $\mathsf k$ if and only if $\rH_{\vert W\vert-i-1}(\Delta|_W;\Z)$ and $\rH_{\vert W\vert-i-2}(\Delta|_W;\Z)$ are torsion-free for every $W\subseteq V$.
\end{proposition}

\begin{proof}
By Remark~\ref{rem:hochster}, given any field $\mathsf k$, the Betti number $\beta_i^\mathsf k\bigl(\mathsf k[\Delta]\bigr)$ is a finite sum of nonnegative integers $\dim_{\mathsf k}\left(\rH_{\vert W\vert-i-1}(\Delta|_W;\mathsf k)\right)$. 
Moreover, we have  $\dim_{\mathsf k}\left(\rH_{\vert W\vert-i-1}(\Delta|_W;\mathsf k)\right) \geq \dim_{\mathbb Q}\left(\rH_{\vert W\vert-i-1}(\Delta|_W;\mathbb Q)\right)$ by Remark~\ref{rem:uct}(c). Hence, $\beta_i^\mathsf k\bigl(\mathsf k[\Delta]\bigr)$ is field independent if and only if every summand is field independent. Thus, applying Remark~\ref{rem:uct}(d) to every summand indexed by $W$, we get that $\beta_i^\mathsf k(\mathsf k[\Delta])$ is independent of $\mathsf k$ if and only if $\rH_{\vert W\vert-i-1}(\Delta|_W;\Z)$ and $\rH_{\vert W\vert-i-2}(\Delta|_W;\Z)$ are torsion-free.
\end{proof}

In view of the above proposition, the field independence of a given Betti number depends on two consecutive reduced homologies for any given $W\subseteq V$. Looking at this from a different perspective, given $W\subseteq V$, we see that torsion in a single reduced homology of the subcomplex induced on $W$ affects two consecutive Betti numbers. In fact, Dalili--Kummini (see \cite[Proposition 5.3]{DaliliKummini}) proved that the Betti table of a monomial ideal over the field $\mathbb Q$ can be obtained from the Betti table over any field by a sequence of consecutive cancellations.

\begin{notation}\label{not:tau}
For $d\geq 0$, define $\tau(d)$ to be the minimum number of vertices of a flag complex $X$ such that $\rH_d(X;\Z)$ contains torsion. If no such flag complex exists, we set $\tau(d)=\infty$.
\end{notation}

We now specialize the preceding criterion to Stanley--Reisner rings of flag complexes. This yields the following corollary, which gives the key condition for the field dependence of Betti numbers, and will be used to answer Katzman's question.

\begin{corollary}\label{cor:tau}
Given any $i \geq 1$, the Betti number $\beta_i^{\mathsf k}(S/I(G))$ depends on $\mathsf k$ for some graph $G$ if and only if
$\tau(d)\le d+i+2$ for some $d\geq 0$.
\end{corollary}
\begin{proof}
By Proposition~\ref{prop:crit}, $\beta_i^{\mathsf k}(S/I(G))$ depends on $\mathsf k$ for some graph $G$ if and only
if there exist $m\in \mathbb N$ and a flag complex $X$  on $m$ vertices having torsion in $\rH_{m-i-1}(X; \mathbb Z)$ or $\rH_{m-i-2}(X; \mathbb Z)$.  Since $\rH_0(X; \mathbb Z)$ is torsion-free for every flag complex $X$, we have $\tau(0)=\infty$. Thus, we may assume $d\geq 1$ throughout in the argument below.

$(\Longrightarrow)$ Assume that there is a flag complex $X$ on $m$ vertices with torsion in $\rH_d(X; \mathbb Z)$, where $d=m-i-1$ or $d=m-i-2$, or equivalently, $m=d+i+1$ or $m=d+i+2$. By the minimality of $\tau(d)$, we have $\tau(d)\le m$. Hence, $\tau(d)\le d+i+1$ or $\tau(d) \leq d+i+2$. We therefore conclude $\tau(d) \leq d+i+2$. 

$(\Longleftarrow)$ Assume that $\tau(d)\le d+i+2$ for some $d \geq 1$, and set $m_0=\tau(d)$. We want to show the existence of a graph $G$ for which $\beta_i^{\mathsf k}(S/I(G))$ depends on the field $\mathsf k$.\\
Let $G_0$ be a graph on $m_0$ vertices such that $\rH_d(\Ind(G_0); \mathbb Z)$ has torsion. Observe that if $G_1$ is the graph obtained from $G_0$ by adding a new vertex $v$ that is adjacent to every vertex of $G_0$, then $\Ind(G_1)$ is the disjoint union of $\Ind(G_0)$ with a point. Since $d \geq 1$, we see that $\rH_d(\Ind(G_0); \mathbb Z) \cong \rH_d(\Ind(G_1); \mathbb Z)$. The graph $G_1$ has one more vertex than $G_0$. Therefore, iterating this process $(d+i+2)-m_0$ times yields a graph $G$ on $m=d+i+2$ vertices such that $\rH_d(\Ind(G); \mathbb Z)$ has torsion. Observe that $d=m-i-2$. Thus, taking $W$ to be the vertex set of $\Ind(G)$ and applying Proposition~\ref{prop:crit}, we get that $\beta_i^{\mathsf k}(S/I(G))$ depends on $\mathsf k$.
\end{proof}

\section{Torsion in the homology of flag complexes}\label{sec:tools}

In this section, we develop the homological and combinatorial tools about flag complexes needed for the proof of the main theorem. We discuss independence complexes under joins and suspensions, vertex domination, a link criterion for torsion via the Mayer--Vietoris sequence, and bounds on the number of vertices supporting homology in a given degree. We begin with known results of Katzman on torsion in flag complexes on a small number of vertices.

\begin{remark}\label{rem:eleven} \hfill{}
\begin{enumerate}
  \item[{\rm (a)}] By \cite[Theorem~4.1]{Katzman}, no flag complex on at most $10$ vertices has torsion in $\rH_d(-;\mathbb Z)$ for any $d$.
  \item[{\rm (b)}] Katzman's classification (see \cite[Appendix~A]{Katzman}) shows that there exist flag complexes on $11$ vertices whose Betti numbers depend on the field, and that for every such complex, the dependence occurs only in $\beta_8$ and $\beta_9$. The existence of such a complex also follows from \cite[Theorem~1.1(ii)]{Adamaszek}.\\
  Let $X$ be a flag complex on $11$ vertices. Taking $W$ to be the full vertex set in Proposition~\ref{prop:crit}, we see that torsion in $\rH_d(X;\mathbb Z)$ forces the field dependence of $\beta_{9-d}$ and $\beta_{10-d}$. Since only  $\beta_8$ and $\beta_9$ are dependent, such an $X$ has torsion exactly in degree $d=1$, and in no other degree. 
  \item[{\rm (c)}] From (a) and (b) above, we get 
  \begin{enumerate}
      \item[{\rm (i)}]  $\tau(1)=11$.
      \item[{\rm (ii)}] No flag complex on eleven vertices has torsion in $\rH_d(X;\mathbb Z)$ for $d\neq 1$.
  \end{enumerate}
\end{enumerate}
\end{remark}

\begin{remark}[see {\cite[Remark~2.2]{Barmak}}]\label{rem:join}
Let $G_1$ and $G_2$ be graphs on disjoint vertex sets $V(G_1)$ and $V(G_2)$. Let $G$ be a graph on the vertex set $V(G_1)\sqcup V(G_2)$. Then $G=G_1\sqcup G_2$ if and only if $\Ind(G)=\Ind(G_1)*\Ind(G_2)$. In particular, if $G_1$ is a single edge, then $\Ind(G)$ is the suspension $\Sigma\bigl(\Ind(G_2)\bigr)$, and hence $\rH_d(\Ind(G);\Z)\cong\rH_{d-1}(\Ind(G_2);\Z)$ for every $d\geq 1$.
\end{remark}

The following observation about the homology of independence complexes of graphs with maximum degree at most two will be useful in what follows.

\begin{remark}\label{rem:deg2}
If every vertex of $G$ has degree at most two, then each connected component of $G$ is a path or a cycle. Thus, by Remark~\ref{rem:join}, $\mathrm{Ind}(G)$ is a join of the complexes of the form $\mathrm{Ind}(P_n)$ and $\mathrm{Ind}(C_n)$. Each such complex is known to be homotopy equivalent to a wedge of spheres (see \cite[Proposition 4.6, Proposition 5.2]{Kozlov}). Since a join of wedges of spheres is again a wedge of spheres, we see that $\widetilde{H}_d(\mathrm{Ind}(G);\mathbb{Z})$ is torsion-free for every $d$.
\end{remark}

\begin{remark}[see {\cite[Lemma~2.4]{Engstrom}}]\label{rem:engstrom}
Let $v\ne u$ be vertices of a graph $G$ with $N_G(v)\subseteq N_G(u)$. Then $\Ind(G)$ collapses onto $\Ind(G\setminus u)$. In particular, $\Ind(G)$ and $\Ind(G \setminus u)$ have isomorphic homology with integer coefficients in every degree.
\end{remark}

We will use the following simple observation concerning graphs with a vertex of degree at most one.
\begin{lemma}\label{lem:mindeg}
Let $G$ be a graph with at least two vertices. If $G$ has a vertex of degree at most one, then one of the following holds: 
\begin{enumerate}
    \item[{\rm (a)}] $G$ has vertices $v\ne u$ such that $N_G(v)\subseteq N_G(u)$.  
    \item[{\rm (b)}] $G$ has a connected component consisting of a single edge. 
\end{enumerate}  
\end{lemma}

\begin{proof}
Let $\deg_G(v)\le1$. If $\deg_G(v)=0$ then $N_G(v)=\emptyset\subseteq N_G(u)$ for any $u\ne v$. So assume $N_G(v)=\{w\}$.\\ If $\deg_G(w)\ge2$, then choosing $u\in N_G(w)\setminus\{v\}$, we have  $u\ne v$ and $w\in N_G(u)$. Thus, $N_G(v)=\{w\}\subseteq N_G(u)$. \\ 
If $\deg_G(w)=1$, then $N_G(w)=\{v\}$, and $\{v,w\}$ is a connected component consisting of a single edge.
\end{proof}

The following lemma shows that if deleting a vertex and taking its link introduce no torsion in the relevant dimensions, then any torsion in the complex forces nontrivial homology in the link.
\begin{lemma}\label{lem:cyc}
Let $X$ be a flag complex on $V$ and $d \geq 1$ be such that $\rH_d(X;\Z)$ has torsion. If $v\in V$ is such that both $\rH_d(X\setminus v;\Z)$ and $\rH_{d-1}(\lk_X(v);\Z)$ are torsion-free, then
$\rH_d(\lk_X(v);\Z)\ne0$.
\end{lemma}

\begin{proof}
Let $X=\Ind(G)$ and $v \in V$, $d \geq 1$ be as given in the statement. We claim that $\lk_X(v) \neq \emptyset$, i.e., $\lk_X(v)$ has at least one vertex. Indeed, if $\lk_X(v)$ has no vertices, then $v$ is adjacent to every other vertex in $G$, and hence $X$ is the disjoint union of $X\setminus v$ and the point $v$. This forces $\rH_d(X;\Z)\cong\rH_d(X\setminus v;\Z)$ for all $d\ge1$. But this is a contradiction to the given hypotheses that $\rH_d(X;\Z)$ has torsion while $\rH_d(X\setminus v;\Z)$ does not. 

Observe that by definitions of $\lk_X(v)$ and $\st_X(v)$, we have $X= (X \setminus v) \cup \st_X(v)$ and $\lk_X(v)=(X\setminus v) \cap \st_X(v)$.
Since $\st_X(v)$ is a cone over $v$, from the Mayer--Vietoris sequence on homology (see \cite[Section 2.2]{Hatcher}), we get the following exact sequence
\[ \rH_d\bigl(\lk_X(v); \mathbb Z\bigr)\xrightarrow{\ f\ }\rH_d(X\setminus v; \mathbb Z) \xrightarrow{\ g\ }\rH_d(X; \mathbb Z)\xrightarrow{\ \partial\ }\rH_{d-1}\bigl(\lk_X(v); \mathbb Z\bigr). \]

Let $0\ne \gamma \in\rH_d(X;\Z)$ be a torsion element. Then $\partial(\gamma)$ is a torsion element of $\rH_{d-1}(\lk_X(v);\Z)$. Since $\rH_{d-1}(\lk_X(v);\Z)$ is torsion-free, we get $\partial(\gamma)=0$. Since $\ker(\partial)= \operatorname{Im}(g)$, we obtain $\gamma\in\operatorname{Im}(g)$. Let, if possible,  $\rH_d(\lk_X(v);\Z)=0$. Then $\ker(g)=\operatorname{Im}(f)=0$, and hence $g$ is injective. This forces $\operatorname{Im}(g) \cong \rH_d(X\setminus v;\Z)$ to be torsion-free. But this is a contradiction since $\operatorname{Im}(g)$ contains the torsion element $\gamma$. Hence, we must have $\rH_d(\lk_X(v);\Z)\ne0$, as required.
\end{proof}

We next consider the minimum number of vertices required for nontrivial homology in a given homological degree, and characterize the extremal case for independence complexes. 
\begin{remark}[see {\cite[Lemma~2.1]{TeraiHibi}}]\label{rem:THfield}
Let $\mathsf k$ be any field, $\Delta$ a flag complex on $N$ vertices, and $n\ge0$. 
\begin{enumerate}
    \item[{\rm (a)}] If $N<2n+2$, then $\rH_n(\Delta;\mathsf k)=0$.
    \item[{\rm (b)}] If $N=2n+2$, then $\rH_n(\Delta;\mathsf k)\ne0$ if and only if $\Delta$ is the join of $n+1$ copies of $S^0$.
\end{enumerate}
\end{remark}

The above remark gives a useful vertex bound, with a precise description of the extremal case. In the next result, we use this to obtain the corresponding statement for homology with integer coefficients of independence complexes.
\begin{lemma}\label{lem:TH}
Let $G$ be a graph on a vertex set $V$ and $\Delta=\Ind(G)$. Then the following hold. 
\begin{enumerate}
    \item[{\rm (a)}] If $\rH_n(\Delta;\Z)\ne0$, then
$\vert V \vert \ge2n+2$. 
\item[{\rm (b)}] If $\vert V \vert =2n+2$ and $\rH_n(\Delta;\Z)\ne0$, then $G$ is a disjoint union of $n+1$ edges. 
\end{enumerate}
\end{lemma}
\begin{proof}
(a) Observe that if $\rH_n(\Delta;\Z)\ne0$, then it contains $\mathbb Z$ as a direct summand or $\Z/p^a\mathbb Z$ as a direct summand for some prime $p$ and $a\in \mathbb N$. If $\mathbb Z$ is a direct summand, then
$\rH_n(\Delta;\mathbb Q)\ne0$, and if $\Z/p^a\mathbb Z$ is a direct summand, then
$\rH_n(\Delta;\mathbb F_p)\ne0$ by Remark~\ref{rem:uct}. Hence, 
by Remark~\ref{rem:THfield}, we get that  if  $\rH_n(\Delta;\Z)\ne0$, then
$\vert V \vert \ge2n+2$. 

(b) If $\vert V \vert =2n+2$ and $\rH_n(\Delta;\Z)\ne0$, then by Remark~\ref{rem:THfield}, $\Ind(G)$ is the join of $n+1$ copies of $S^0$. By Remark~\ref{rem:join}, this is equivalent to $G$ being a disjoint union of $n+1$ edges.
\end{proof}

\section{The main theorem}\label{sec:main}

We are now ready to prove the main theorem of the article. Building on the results of the previous section, we obtain the lower bound 
$\tau(d)\geq d+10$ on the number of vertices needed for torsion, and then deduce that the seventh Betti number is field independent. This answers Katzman's question in the affirmative.

\begin{theorem}\label{thm:main}
Given any $d \geq 0$, we have $\tau(d)\geq d+10$.
\end{theorem}
\begin{proof}
For the sake of contradiction, assume that there exists $d \geq 0$ for which $\tau(d) \leq d+9$. Let $m$ be the least integer with the property that there exists a flag complex $X$ on $m$ vertices and an integer $d$ such that $\rH_d(X;\Z)$ has torsion and $m\le d+9$. Then $X=\Ind(G)$ for some graph $G$ on $m$ vertices. Let $V$ denote the vertex set of $G$.

By Remark~\ref{rem:eleven}, we get $m \geq 11$. Hence $m\le d+9$  gives $d\ge m-9\ge2$. We claim that $2 \leq \deg_G(v) \leq m-2d-3$ for all $v \in V$. If $G$ has a vertex of degree at most one, then by Lemma~\ref{lem:mindeg}, we have two cases: 

\textit{Case 1:} There exist $u, v \in V$ with $u \neq v$ such that $N_G(v) \subseteq N_G(u)$. \\
In this case, by Remark~\ref{rem:engstrom}, $\Ind(G\setminus u)$ is a flag complex on $m-1$ vertices having torsion in $\rH_d(\Ind(G\setminus u); \Z)$. But this contradicts the minimality of $m$ since $m-1 \leq d+8\leq d+9$.

\textit{Case 2:} $G$ has a connected component $G_1$ consisting of a single edge.\\
Let $G= G_1 \sqcup G_2$. Then by Remark~\ref{rem:join}, we see that $\Ind(G_2)$ is a flag complex on $m-2$ vertices having torsion in $\rH_{d-1}(\Ind(G_2) ; \Z)$. This is again a contradiction to the minimality of $m$ since $m-2 \leq (d-1)+8 \leq (d-1)+9$.

Therefore, by the two cases above, we get $2 \leq \deg_G(v)$ for all $v \in V$. To prove the remaining inequality, consider any $v \in V$. By Remark~\ref{eq:basic}, $X\setminus v=\Ind(G\setminus v)$ and $\lk_X(v)=\Ind(G \setminus N_G[v])$ are flag complexes on $m-1$ and $m-1-\deg_G(v)$ vertices, respectively. Note that since $m \leq d+9$, we have $m-1 \leq d+9$ and $m-1-\deg_G(v) \leq (d-1)+9$. 
Thus, by the minimality of $m$, we get that $\rH_d(X\setminus v;\Z)$ and $\rH_{d-1}(\lk_X(v);\Z)$ are torsion-free. By Lemma~\ref{lem:cyc}, we see that $\rH_d\bigl(\lk_X(v);\Z\bigr)\ne0 $.
Applying Lemma~\ref{lem:TH} to $\lk_X(v)$, we get $m-1-\deg_G(v)\ge2d+2$, i.e., $\deg_G(v) \leq m-2d-3$. This completes the proof of the claim.

By the claim, in particular, we obtain $2 \leq m-2d-3$, i.e., $2d+5 \leq m$. Using $m \leq d+9$, we get $2d+5 \leq d+9$, i.e., $d \leq 4$. Hence we have $2 \leq d \leq 4$. From $11\leq m \leq d+9$, we see that if $d=2$, then $m=11$. Similarly, if $d=3$, then $m \in \{11, 12\}$. Finally, if $d=4$, then $2d+5 \leq m$ gives $13\leq m$, and $m\leq d+9$ gives $m\leq 13$. Therefore, $m=13$. We therefore have the following possibilities for $(d,m)$: $(2,11), (3,11), (3, 12), (4, 13)$. 
To complete the proof of the theorem, we now show that none of these possibilities can occur.

If $(d,m)=(2,11)$, then $X$ is a flag complex on $11$ vertices with torsion in $\rH_2(X; \Z)$, which is impossible by Remark~\ref{rem:eleven}.

Let $(d,m) \in \{(3,11), (4,13)\}$.  Then $m-2d-3=2$. Hence, from $2 \leq \deg_G(v)\leq m-2d-3$, we get $\deg_G(v)\le2$ for all $v \in V$. By Remark~\ref{rem:deg2}, $\rH_n(X;\Z)$ is torsion-free for every $n$. Hence, $(d,m)\notin \{(3,11), (4,13)\}$.

Finally, let $(d,m)=(3,12)$. Then we have $2 \leq \deg_G(v) \leq m-2d-3=3$ for all $v \in V$. If $\deg(v)=2$ for all $v \in V$, then Remark~\ref{rem:deg2} applies and we get a contradiction as before. Therefore, there exists $v\in V$ such that $\deg_G(v)=3$. 

We have $\vert N_G(v) \vert =3$ and $ \vert V\setminus N_G[v] \vert = 8$. As observed earlier in this proof, we have $\rH_3(\lk_X(v); \Z)\ne0$. 
Also, by Remark~\ref{eq:basic},  $\lk_X(v)=\Ind(G\setminus N_G[v])$ is a flag complex on exactly $8=2\cdot3+2$ vertices. Hence, applying Lemma~\ref{lem:TH} with $\Delta=\Ind(G\setminus N_G[v])$, we get that  $G\setminus N_G[v]$ is a disjoint union of $4$ edges.

Since every $u \in G \setminus N_G[v]$ satisfies $\deg(u) \geq 2$ and since $u$ is not a neighbour of $v$, there exists an edge between $u$ and some vertex in $N_G(v)$. Since there are $8$ vertices in $G \setminus N_G[v]$, we have at least $8$ edges having one vertex in $G \setminus N_G[v]$ and other in $N_G(v)$. Moreover, each vertex in $N_G(v)$ is connected to $v$. This forces $\sum_{w \in N_G(v)} \deg_G(w) \geq 8+3 =11$. On the other hand, since every vertex in $G$ has degree at most $3$ and $\vert N_G(v)\vert=3$, we must have $\sum_{w \in N_G(v)} \deg_G(w)\leq 9$. This contradiction implies that the possibility $(d, m) = (3,12)$ also does not occur, and the proof is complete.
\end{proof}

Recall that Corollary~\ref{cor:tau} gives a criterion for the field independence of $\beta_i^{\mathsf k}(S/I(G))$ in terms of a bound on $\tau(d)$. Substituting the bound on $\tau(d)$ obtained in the theorem above gives us the following. 

\begin{corollary}\label{cor:beta7}
For every graph $G$ and every field $\mathsf k$, the Betti number $\beta_7^\mathsf k(S/I(G))$ is
independent of $\mathsf k$. 
\end{corollary}
\begin{proof}
    Setting $i=7$ in Corollary \ref{cor:tau}, we see that $\beta_7^\mathsf k(S/I(G))$ depends on $\mathsf k$ for some graph $G$ if and only if $\tau(d)\leq d+9$ for some $d$. But this is impossible by Theorem \ref{thm:main}. 
\end{proof}
The above corollary answers the question of Katzman in affirmative.

\end{document}